\documentclass[sort&compress]{elsarticle}
\usepackage{amssymb, amsmath, amsthm}
\usepackage[all]{xy}
\usepackage{cite}
\usepackage{url}
\usepackage{graphicx}

\theoremstyle{plain}
\newtheorem{theorem}{Theorem}
\newtheorem*{theorem*}{Theorem}
\newtheorem{lemma}[theorem]{Lemma}
\newtheorem{prop}[theorem]{Proposition}
\newtheorem{corol}[theorem]{Corollary}

\theoremstyle{definition}
\newtheorem{defin}[theorem]{Definition}

\newtheorem{example}[theorem]{Example}

\theoremstyle{remark}
\newtheorem{remark}[theorem]{Remark}

\newcommand\hide[1]{}
\newcommand\cat\mathcal
\newcommand\set[1]{\left\{#1\right\}}

\newcommand{\downmono}{\mathbin{\rotatebox[origin=c]{-90}{$\rightarrowtail$}}}
\newcommand{\downepi}{\mathbin{\rotatebox[origin=c]{-90}{$\twoheadrightarrow$}}}
\newcommand\exlex{_\mathrm{ex/lex}}\newcommand\exreg{_\mathrm{ex/reg}}

\begin{document}
\begin{frontmatter}
\title{Resolvent embeddings}

\author[W. P. Stekelenburg]{Wouter Pieter Stekelenburg}
\address{Faculty of Mathematics, Informatics and Mechanics,
University of Warsaw,
Banacha 2,
02-097 Warszawa,
Poland}
\ead{w.p.stekelenburg@gmail.com}
\fntext[W. P. Stekelenburg]{Corresponding author; tel.: +31624543216}

\begin{abstract}
Many ex/reg completions $J:\cat C\to\cat C\exreg$ that arise in categorical realizability and tripos theory admit left Kan extensions of arbitrary finitely continuous functors to arbitrary exact categories. This paper identifies the property which is responsible for these extensions: the functors are \emph{resolvent}. Resolvency is characteristic of toposes that are ex/reg completions of regular categories with (weak) dependent products and generic monomorphisms. It also helps to characterize the toposes that the tripos-to-topos construction produces.
\end{abstract}

\begin{keyword}
exact completion\sep topos\sep tripos\sep Kan extension\sep generic monomorphism\MSC{18A22}
\end{keyword}

%AMS 18A22

\end{frontmatter}

\newcommand\RT{\mathsf{RT}}
\newcommand\Set{\mathsf{Set}}
\newcommand\pow{\mathbf P}
\newcommand\Eff{\mathsf{Eff}}
\newcommand\Asm{\mathsf{Asm}}

The effective topos $\Eff$ is the ex/reg completion of the category of assemblies $\Asm$. This means that for each exact category $\cat E$ every regular functor $\Asm \to \cat E$ factors in an up to isomorphism unique way through a regular embedding $\Asm\to \Eff$. Although this is not always the case with ex/reg completions, functors $\Asm\to\cat E$ which only preserve finite limits also factor through the embedding $\Asm \to \Eff$. The reason is that the embedding is \emph{resolvent}, a property this paper defines and explains.

\section{Introduction}
Each regular category has two exact completions: the \emph{ex/lex completion} $I:\cat C\to \cat C\exlex$ and the \emph{ex/reg completion} $J:\cat C\to\cat C\exreg$. The ex/lex completion is a free finitely continuous functor into an exact category, while the ex/reg completion is a free \emph{regular} functor into an exact category. The restriction in freedom from the first to the second induces a unique regular functor $K:\cat C\exlex\to\cat C\exreg$ which satisfies $KI\simeq J$.

The embedding $J:\Asm \to \Asm\exreg\simeq \Eff$ is special because it is \emph{resolvent} (definition \ref{resolve}). Theorem \ref{ThA} explains what this means:

\begin{theorem*}[\ref{ThA}]
For every regular category $\cat C$, the following three conditions are equivalent:
\begin{enumerate}[1.]%font not loadable?
\item the functor $J:\cat C \to\cat C\exreg$ is resolvent;
\item the canonical functor $K:\cat C\exlex\to\cat C\exreg$ has a right adjoint right inverse;
\item for each exact $\cat E$, each finitely continuous functor $F:\cat C \to \cat E$ has a left Kan extension along $J:\cat C \to \cat C\exreg$.
\end{enumerate}
\end{theorem*}

Thus finite limit preserving functors from $\Asm$ to exact $\cat E$ extend to finite limit preserving functors $\Eff\to \cat E$.

The causes of resolvency are the (weak) dependent products and the generic monomorphisms of $\Asm$. Weak dependent products in $\cat C$ make $\cat C\exlex$ a locally Cartesian closed category. The following notion helps to explain the role of the generic monomorphism.

\begin{defin} A monomorphism $m$ of $\cat C\exlex$ is \emph{open} if $m$ is right orthogonal to $Ie$ for each regular epimorphism $e$ of $\cat C$.
\end{defin}

The generic monomorphism induces a classifier $t:1\to \Omega$ for open monomorphisms. This object $\Omega$ is \emph{local}, which means that $\cat C\exlex(I-,\Omega)$ is a sheaf for the \emph{regular topology} on $\cat C$. Local objects form a subtopos $\cat S$ of $\cat C\exlex$, %het geometrische morfisme bestaat!
such that $I:\cat C \to \cat C\exlex$ factors through the embedding $\cat S \to \cat C\exlex$ by a regular functor. The right adjoint right inverse of $K$ comes from the universal properties of the functors $I$ and $J$.

It is not hard to see that $\cat C$ has weak dependent products and a generic monomorphism if $J:\cat C \to \cat C\exreg$ is resolvent and if $\cat C\exreg$ is a topos, so there is an equivalence.

\begin{theorem*}[\ref{ThB}] For each regular category $\cat C$ the following are equivalent:
\begin{enumerate}
\item $\cat C$ has weak dependent products and a generic monomorphism;
\item $\cat C\exreg$ is a topos and $J:\cat C\to\cat C\exreg$ is resolvent.
\end{enumerate}
\end{theorem*}

Section \ref{TriThe} explores an important connection with \emph{tripos theory}. The \emph{tripos-to-topos construction} turns each tripos $T$ on each category with finite products $\cat B$ into topos $\cat B[T]$ together with a finitely continuous functor $\nabla:\cat B\to\cat B[T]$. The topos $\cat B[T]$ is the ex/reg completion of a regular full subcategory $\Asm(T)$ (of $T$-assemblies) which has weak dependent products and a generic monomorphism. Morphisms of triposes induce finitely continuous functors between the toposes they represent \citep[see][]{MR2479466,Pittsthesis, MR578267,a2CAotTtTC} because the embedding $\Asm(T) \to \cat B[T]$ is resolvent.

\subsection{Related literature}
The exact completion constructions which play an important role throughout this paper come from Aurelio Carboni's works \citep{MR1600009, MR1358759, MR948482, MR678508,MR1787592}. 

This paper extends Mat\'ias Menni's analysis of toposes that are exact completions \citep{MR1900904, Menni00exactcompletions, MR1948025, MR1870615, MR2320014}. Menni proved that the ex/reg completion of a locally Cartesian closed category regular category with generic monomorphism is a topos. I show that weakly Cartesian closed works too and add a characterization of toposes of this form.

It does not look like regular categories need a generic monomorphism to turn their ex/reg completions into toposes, but I have not found an example. 

Ex/reg completions play a role in the theory of realizability toposes, each of which is an ex/reg completion of a locally Cartesian closed with a generic monomorphism, like Hyland's \emph{effective topos} $\Eff$ \citep{MR717245, MR2479466} is an ex/reg completion of $\Asm$. An alternative way to study these toposes is through \emph{tripos theory} \citep{MR2479466,Pittsthesis, MR578267,a2CAotTtTC}. In fact the central idea in this paper--\emph{resolvent functors}--connects to the \emph{weakly complete objects} of tripos theory. 

\newcommand\Fam{\mathsf{Fam}}
I have only started to work out the connections with other constructions of exact categories. If $\cat E$ is exact, then a resolvent embedding $F:\cat C\to\cat E$ forces $\cat C$ to have weak finite limits. Therefore finite completeness probably is a stronger than necessary assumption for several results in this paper. For another example: let $\Fam(\cat C)$ be the category of families in $\cat C$ (see example \ref{famex} for the definition) and let $\cat E$ be a Grothendieck topos; a coproduct preserving resolvent embedding $\Fam(\cat C) \to \cat E$ implies that $\cat E$ is the topos of sheaves for some coverage on $\cat C$. This is a reformulation of Giraud's theorem.

\hide{%eventuele toepassingen doen me nu niets.
\hide{Hyland, Pitts, Frey}
My paper stays within the sandbox of finitely complete categories, but exact completions are defined for categories which are finitely complete in a weakened sense . Moreover, Michael Shulman shows that both the ex/lex and the ex/reg completion are special cases of the exact completion of a \emph{unary site} in \citet{ECnSS}, which in turn is a member of a family of colimit completion construction for sites which includes the topos of sheaves. There is much to explore here concerning completions which are toposes.

Finally, the ex/lex completion of a category has some parallels to the construction of the category of Kan complexes in the same category. In this analogy the generic monomorphism is a generic small Kan fibration. Interesting \emph{$\infty$-toposes} and other models of \emph{homotopy type theory} may arise this way.
}

\section{Preliminaries} %wikipedia is voorkennis
The major part of this paper concerns \emph{finitely complete categories} and \emph{finitely continuous functors} between these categories, i.e.\ categories with finite limits and functors that preserve those limits. Of these the regular and exact (finitely complete) categories play a special role.

\begin{remark} In this paper a category \emph{has} certain limits and colimits, if there are functors and natural transformations which provide limit cones and colimit cocones. For small categories this is equivalent to the axiom of choice. \end{remark}

I regularly refer to regular epimorphisms $e:X\to Y$ as \emph{covers} of $Y$.
Exact completions are ways to add quotient objects to a categories in order to make them exact.

\newcommand\FC{\mathsf{FC}}
\newcommand\Reg{\mathsf{Reg}}
\newcommand\Ex{\mathsf{Ex}}
\begin{defin}
Let $\FC$ be the (2,1)-category of finitely complete categories, finitely continuous functors and natural isomorphisms of functors. Let $\Reg$ be the sub-(2,1)-category of regular categories and functors and let $\Ex$ be the sub-(2,1)-category exact categories and regular functors. For each finitely complete category $\cat C$ an \emph{ex/lex completion} is an exact category $\cat C\exlex$ with a finitely continuous functor $I:\cat C \to \cat C\exlex$ which induces an equivalence of groupoids $\Ex(\cat C\exlex,\cat E) \to \FC(\cat C,\cat E)$. For each regular $\cat D$ an \emph{ex/reg completion} is an exact category $\cat D\exreg$ with a regular functor $I:\cat D \to \cat D\exlex$ which induces an equivalence of groupoids $\Ex(\cat C\exlex,\cat E) \to \Reg(\cat C,\cat E)$.
\end{defin}

\newcommand\id{\mathrm{id}}
Several useful properties of completions follow.

\begin{lemma} For every object $X$ of $\cat C\exlex$ there is an object $X_0$ of $\cat C$ together with a regular epimorphism $c_X:IX_0 \to X$. Moreover, each regular epimorphism $X\to IY$ is split. For every object $X$ of $\cat C\exreg$ there is an object $X_0$ of $\cat C$ together with a regular epimorphism $c_X:JX_0 \to X$.\label{proj} \end{lemma}

\begin{proof} See \citet{MR1600009}. \end{proof}

\section{Resolvent functors}
Certain full subcategories of an exact category $\cat E$ allow the representation of each object of $\cat E$ by a kind of generalized congruence called a \emph{pseudoequivalence relation} and each morphism by an equivalence class of morphisms of pseudoequivalence relations. Since each finitely continuous functor $F$ defined on this subcategory preserves the representations and their equivalences, one can build the left Kan extension of $F$ to the whole category by sending a representation $R$ to the quotient of $FR$, provided that the codomain of $F$ has enough quotient objects. This section introduces and explains the required property of the subcategory as a condition on the inclusion functor.

\subsection{Pseudoequivalence relations} 
Pseudoequivalence relations are an important class of parallel pairs of arrows, which have coequalizers in every exact category. They are also the workhorses of this section.

\begin{defin} A parallel pair of arrows $(d_0,d_1):X_1 \rightrightarrows X_0$ is a \emph{pseudoequivalence relation} or \emph{trivial bigroupoid} if there are morphisms $r:X_0\to X_1$, $s:X_1\to X_1$ and $t:X_1\times_{X_0} X_1\to X_1$ which satisfy $d_0\circ r = d_1\circ r = \id_{X_0}$, $(d_0,d_1)\circ s = (d_1,d_0)\circ s$, $(d_0,d_1)\circ t = (d_0\circ \pi_0, d_1\circ \pi_1)$.

\[ \begin{array}{ccc} 
\xymatrix{
 X_0 \ar[r]^r\ar[d]_r\ar[dr]|{\id_{X_0}} & X_1\ar[d]^{d_0} \\
 X_1 \ar[r]_{d_1} & X_0
}&
\xymatrix{
 X_1 \ar[r]^s \ar[d]_{d_0} \ar[dr]|{d_1} & X_1 \ar[dl]|{d_1} \ar[d]^{d_0} \\
 X_0 & X_0
}
\xymatrix{
 X_1\ar[d]_{d_0} & X_1\times_{X_0} X_1 \ar[l]_(.6){\pi_0}\ar[r]^(.6){\pi_1} \ar[d]|t & X_1 \ar[d]^{d_1} \\
 X_0 & X_1\ar[l]^{d_0}\ar[r]_{d_1} & X_0
}
\end{array}\]

A \emph{morphism of pseudoequivalence relations} $(X_0,X_1,d_0,d_1) \to (Y_0,Y_1,d'_0,d'_1)$ is simply a pair of morphisms $f_0: X_0\to X_0$ and $f_1:X_1\to X_1$ such that $d_i \circ f_1 = f_0 \circ d_i$ for $i=0,1$.
\[ \xymatrix{
 X_1 \ar[r]^{f_1}\ar[d]_{d_i} & X_1\ar[d]^{d_i} \\
 X_0 \ar[r]_{f_0} & X_0
} \]

Two parallel morphisms $(f_0,f_1),(g_0,g_1):(X_0,X_1,d_0,d_1) \to (Y_0,Y_1,d'_0,d'_1)$ are \emph{equivalent} if there is a map $\eta:X_0 \to Y_1$ such that $(d'_0,d'_1)\circ\eta = (f_0, g_0)$.
\[ \xymatrix{
 X_0 \ar[r]^{f_0}\ar[d]_{g_0}\ar[dr]|\eta & Y_0 \\
 Y_0 & Y_1 \ar[u]_{d'_0}\ar[l]^{d'_1}
} \]

\end{defin}

\begin{lemma} Exact categories have coequalizers for pseudoequivalence relations. Morphisms of pseudoequivalence relations induce morphisms between coequalizers. Equivalent morphisms induce the equal morphisms of the coequalizers. \end{lemma}

\begin{proof} If $(X_0,X_1,d_0,d_1)$ is a pseudoequivalence relation, then the image of $(d_0,d_1):X_1\to X_0\times X_0$ is a congruence. This congruence has a coequalizer by exactness and the coequalizer of the congruence is the coequalizer of the pseudoequivalence relation.

If $(f_0,f_1):(X_0,X_1,d_0,d_1) \to (Y_0,Y_1,d'_0,d'_1)$ is a morphism, $x:X_0 \to X$ is the coequalizer of $d_0,d_1$ and if $y:Y_0 \to Y$ is the coequalizer of $d'_0$ and $d'_1$, then $y\circ f_0\circ d_0 = y\circ d'_0 \circ f_1 = y\circ d'_1\circ f_1 = y\circ f_0 \circ d_1$ and therefore $y\circ f_0$ factors uniquely through $x$.

\[ \xymatrix{
X_1 \ar[d]_{f_1} \ar@<-1ex>[r]_{d_0}\ar@<1ex>[r]^{d_0} & X_0\ar[d]^{f_0} \ar[r]^x & X\ar@{.>}[d] \\
Y_1 \ar@<-1ex>[r]_{d'_0}\ar@<1ex>[r]^{d'_0} & Y_0 \ar[r]_y & Y \\
}\]

If $\eta$ is an equivalence of $(f_0,f_1)$ and $(g_0,g_1)$, then $f\circ x = y\circ f_0  = y \circ d'_0\circ \eta = y\circ d'_1\circ \eta = y\circ g_0 = g\circ x$. Since $x$ is an epimorphism $f=g$.
\hide{\[ \xymatrix{
Y_1 \ar[rrr]_{d'_0}\ar[dd]^{d'_1} &  && Y_0\ar[dd]^y \\
& X_0 \ar[urr]^{f_0} \ar[ul]|\eta \ar[dl]_{g_0}\ar[r]^x & X \ar@<1ex>[dr]^f\ar@<-1ex>[dr]_g\\
Y_0 \ar[rrr]_y&  && Y
}\]}
\[ \xymatrix{
& Y_0\ar@/^3ex/[drr]^y \\
Y_1\ar[ur]^{d'_0}\ar[dr]_{d'_1} & X_0\ar[l]_\eta \ar[u]_{f_0}\ar[d]^{g_0} \ar[r]^x & X \ar@<1ex>[r]^f\ar@<-1ex>[r]_g & Y \\
& Y_0\ar@/_3ex/[urr]_y
}\]
\end{proof}

\subsection{Resolvent functors}
This subsection introduces \emph{resolvent functors} and should that if the inclusion of a subcategory $\cat A$ of an exact category $\cat B$ is resolvent, then every object of $\cat B$ is the equalizer of some pseudoequivalence relation in $\cat A$.

This subsection introduces the quality the inclusion of a subcategory of an exact category must have, in order to represent every morphism as a morphism of pseudoequivalence relations.

\begin{defin} Let $F:\cat C\to\cat D$ be a arbitrary functor. For each object $D$ of $\cat D$, a \emph{resolution} is an object $D_0$ of $\cat D$ and a morphism $c_D:FD_0 \to D$ with the following two properties.
\begin{enumerate} 
\item The morphism $c_D$ is the coequalizer of its own kernel pair.
\item For each morphism $f:FC\to D$ there is an $f_0:C\to D_0$ such that $f=c_D\circ f_0$. 
\[ \xymatrix{
& FD_0 \ar[d]^{c_D} \\
FC\ar[r]_f \ar@{.>}[ur]^{Ff_0} & D
}\]
\end{enumerate}
A functor $F:\cat C\to\cat D$ is \emph{resolvent} if each object $D$ of $\cat D$ has a resolution.\label{resolve}
\end{defin}

\begin{remark} I assume that there is a function $f\mapsto f_0$ such that $f=c_D\circ If_0$ for all $f:IC\to D$ for the purpose of constructing some functors below. \end{remark}

\begin{example} In an ex/lex completion $I:\cat C \to\cat C\exlex$ the functor $I$ is resolvent by lemma \ref{proj}. \end{example}

\newcommand\Sh{\mathsf{Sh}}
\begin{example} The topos of sheaves $\Sh(H)$ over a complete Heyting algebra $H$ has a subcategory of subsheaves of constant sheaves, which I denote by $\Asm(H)$. The inclusion $\Asm(H) \to\Sh(H)$ is resolvent. \end{example}

\newcommand\dual{^{op}}
\begin{example} Let $\cat C$ be a small regular category. Let $\Fam(\cat C)$ be the category of families in $\cat C$. Objects in $\Fam(\cat C)$ are pairs $(I,f)$ where $I$ is a set and $f:I\to \cat C_0$ assigns an object of $\cat C$ to each member of $I$. A morphism $(I,\alpha) \to (J,\beta)$ is a function $f:I\to J$ together with a function $\phi:I \to \cat C_1$ which assigns a morphism $\phi(i):\alpha(i) \to \beta(f(i))$ to each $i\in I$. Because $\cat C$ is regular, $\Fam(\cat C)$ is a regular category and $\Fam(\cat C)\exreg$ is the topos $\Sh(\cat C)$ of sheaves for the regular topology of $\cat C$. The embedding $\Fam(\cat C)\to\Sh(\cat C)$ is resolvent; for each sheaf $S:\cat C\dual \to \Set$ the resolution comes from the bundle $\coprod_{X\in \cat C_0} SX \to \cat C_0$ of local sections. \label{famex}
\end{example}

The resolvent embeddings are connected to the pseudoequivalence relations in the following way.

\begin{lemma} If $\cat C$ is finitely complete, $\cat D$ is regular and $F:\cat C \to\cat D$ is a resolvent embedding, then every object of $\cat D$ is the coequalizer of a pseudoequivalence relation in $\cat C$ and morphisms of pseudoequivalence relations in $\cat C$ induce every morphism of $\cat D$. \end{lemma}

\begin{proof} 
For each object $X$ of $\cat D$ a resolution $c_X:FX_0 \to X$ provides $X_0$. The resolution $c_X$ has a kernel pair $(p,q):W\to FX_0$ and $W$ has its own resolution $c_W:FX_1 \to W$. By fullness there is a parallel pair of arrows $d_0,d_1:X_1\to X_0$ such that $Fd_0 = p\circ c_W$ and $Fd_1 = q\circ c_W$. The reader may check that these form a pseudoequivalence relation.

Let $d'_0,d'_1:Y_1 \to Y_0$ be a pseudoequivalence relation and let $c_Y:FY_0 \to Y$ be the coequalizer of $Fd'_0$ and $Fd'_1$. For each morphism $f:X\to Y$, $f\circ c_X:FX_0 \to Y$ factors through $c_Y:FY_0 \to Y$. Because $f_0$ commutes with the congruences determined by the kernel pairs, there is an $f_1:X_1\to Y_1$ which turns $(f_0,f_1)$ into a morphism of pseudoequivalence relations. Finally, if $g_0:X_0\to Y_0$ satisfies $c_Y\circ Ig_0 = f\circ c_X$, then $(f_0,g_0): FX_0 \to FY_0\times FY_0$ factors through $(d'_0,d'_1):FY_1\to FY_0\times FU_0$ inducing an equivalence.

Note that faithfulness makes the diagrams in $\cat C$ commute.
\end{proof}

\begin{defin} If $F:\cat C\to\cat D$ is a left exact functor, $(X_0,X_1,d_0,d_1)$ is a pseudoequivalence relation and $c_X:FX_0 \to X$ is a coequalizer of $Fd_0$ and $Fd_1$, then $(X_1,X_0,d_1,d_0)$ \emph{represents} $X$. Similarly, if $(f_0,f_1)$ is a morphism $(X_0,X_1,d_0,d_1)\to (Y_0,Y_1,d'_0,d'_1)$, if $c_X:FX_0 \to X$ and $c_Y:FY_0 \to Y$ are coequalizers of $d_0$, $d_1$ and $d'_0$, $d'_1$ respectively, and if $f\circ c_X = c_Y\circ f_0$, then $(f_0,f_1)$ \emph{represents} $f$. \end{defin}

\subsection{Kan extensions}
Resolvency is a variation of Freyd's \emph{solution set condition}. Instead of rights adjoints, the condition induces left Kan extensions of certain functors. Throughout this subsection $\cat C$ and $\cat D$ are finitely complete, $F:\cat C\to\cat D$ is a finitely continuous resolvent full and faithful functor and $\cat E$ is exact.

\begin{lemma} Every finitely continuous $G:\cat C\to\cat E$ has a left Kan extension $F_!(G)$ along $F$. Moreover, the unit of the adjunction $F_!\dashv F^*$ is a natural isomorphism. \label{LKE}
\end{lemma}

\begin{proof} Let $(X_1,X_0,d_1,d_0)$ be a pseudoequivalence relation in $\cat C$ such that $X$ is the coequalizer of $Fd_0$ and $Fd_1$. The image $(GX_0,GX_1,Gd_1,Gd_0)$ has a coequalizer in $\cat E$. So let $F_!(G)(X) = GX_0/GX_1$ be this coequalizer. For any morphism $f:X\to Y$, there are morphisms of pseudoequivalence relations $(f_0,f_1)$ in $\cat C$ to represent it and $(Gf_0,Gf_1)$ determines a unique morphism $F_!(G)(X) \to F_!(Y)$. Let $F_!(G)(f)$ be that morphism. Since equalities between morphisms lift to equivalences of morphisms of pseudoequivalences and since equivalent morphisms of pseudoequivalences induce the same morphisms of coequalizers, $F_!(G)$ preserves compositions and identities. Hence $F_!(G)$ is a functor.

This functor is finitely continuous for reasons I only sketch here. Clearly, $F_!(G)$ preserves the terminal object $1$. There is a weak pullback construction for pseudoequivalence relations which induces pullbacks of coequalizers. Finally, $\id_X:X\to X$ makes sure that any two representation of the same object are connected by pairs of morphisms of pseudoequivalence which are inverses up to equivalence, so pullbacks commute with weak pullbacks up to equivalence.

Concerning the left Kan extension part: there is an isomorphism $\eta:G \to F_!(G)F$ because each object of $\cat C$ represents itself in a trivial way. Hence every natural transformation $F_!(G) \to H$ induces a unique transformation $G \to HF$. The isomorphism $\eta$ is the unit. Suppose $H:\cat D\to\cat E$ is a finitely continuous functor and $\nu:G\to HF$ some natural transformation. For each object $X$ and each resolution $h:FX_0\to X$, $Hh\circ \nu_{X_0}\circ Gd_0 = Hh\circ Hd_0\circ \nu_{X_1} = Hh\circ Hd_1\circ \nu_{X_1} = Hh\circ \nu_{X_0}\circ Gd_1$ and since $F_!(G)(X)$ is the coequalizer of $Gd_0$ and $Gd_1$ there is a unique morphism $\mu_X: F_!(G)X\to HX$. 
\[\xymatrix{
GX_1 \ar[rr]^{Gd_0}\ar[dd]_{Gd_1}\ar[dr]^{\nu_{X_1}} && GX_0 \ar[dr]^{\nu_{X_0}}\ar[dd] \\
 & HFX_1 \ar[rr]^(.3){HFd_0}\ar[dd]_(.3){HFd_1} && HFX_0 \ar[dd]^{Hh} \\
GX_0\ar[dr]_{\nu_{X_0}}\ar[rr] && F_!(G)X \ar[dr]^{\mu_X} \\
& HFX_0\ar[rr]_{Hh} && HX
}\]
By generalization each natural transformation $\nu:G\to HF$ determines a unique transformation $\mu: F_!(G) \to H$. Uniqueness enforces the naturalness of $\mu$. Therefore $F_1(G)$ is the left Kan extension of $G$ along $F$.
\end{proof}

The left Kan extensions are finitely continuous functors which send resolutions to regular epimorphisms. They do not always preserve other regular epimorphisms. This is characteristic of functors which occur as left Kan extensions.

\begin{lemma} If $H:\cat D\to \cat E$ is finitely continuous, then the counit $\epsilon_H:F_!(HF)\to H$ is an isomorphism if and only if $H$ sends resolutions to regular epimorphisms. \end{lemma}

\begin{proof} If $(X_0,X_1,d_0,d_1)$ is an arbitrary pseudoequivalence representing $X$, then both $H$ and $F_!(H)$ send $X$ to the quotient $HFX_0/HFX_1$. \end{proof}

There is a converse.

\begin{lemma} Let $\cat D$ be exact, let $F:\cat C\to \cat D$ be fully faithful and finitely continuous and suppose that for each exact $\cat E$, each finitely continuous $G:\cat C\to \cat E$ has a left Kan extension $F_!(G)$ along $F$ and that $F_!(G)F\simeq G$. Then $F$ is resolvent.\label{converse} \end{lemma}

\newcommand\cod{\mathrm{cod}}
\begin{proof} Define the category $(F\downepi \cat D)$ as follows. The objects are regular epimorphisms $e:FX \to Y$. A morphism $e\to e'$ is a morphism $f:\cod e \to \cod e'$ such that $f\circ e$ factors through $e'$.

The category $(F\downepi \cat D)$ is exact. Let $G:\cat C\to (F\downepi \cat D)$ be the functor that maps $X$ to the trivial cover $\id_FX:FX\to FX$. For each object $X$ of $\cat D$, $F_!(G)X$ is a cover $FX_0 \to X$. If $f:FY\to X$ is any morphism, then $F_!(G)f: F_!(G)FY\simeq GY \to F_!(G)X$ forces $f$ to factor through $F_!(G)X$. Hence $F_!(G)(X)$ is a resolution. By generalization $F$ is a resolvent functor.
\end{proof} %fuck!

So in the specific case of finitely continuous embeddings into exact categories, resolvency is equivalent to the property of admitting left Kan extensions.

\subsection{Ex/reg completions}
There is an important connection between resolvency and the ex/lex completion. Let $\cat D$ be exact and let $F:\cat C \to \cat D$ be a finitely complete resolvent embedding. Since $F$ is finitely complete, there is an up-to-isomorphism-unique regular functor $K:\cat C\exlex \to \cat D$ such that $KI\simeq F$. Because $F$ is resolvent, the left Kan extension $H=F_!(I):\cat D\to\cat C\exlex$ also exists. The functors $K$ and $H$ are connected in the following way.

\begin{lemma} There is an adjunction $K\dashv H$. Moreover $H$ is fully faithful. \label{adjunct} \end{lemma}

\begin{proof} \hide{Apply the adjoint functor theorem to $K$!  Can't! $K$ is not fully faithful! }
Note that lemma \ref{LKE} implies $HF\simeq I$. 
The functor $\id_{\cat C\exreg}$ sends resolutions to regular epimorphisms because resolutions are regular epimorphisms. For that reason, $F_!(F)\simeq \id_{\cat D}$. 
Since $K$ is regular, $KH$ is also a functor that sends resolutions to regular epimorphisms. Because $KHJ\simeq KI\simeq J$, $J_!(J) \simeq J_!(KHJ)\simeq KH$ and $KH\simeq \id_{\cat C\exreg}$. Conclusion: there is a canonical isomorphism that serves as co-unit $KH \to \id_{\cat C\exreg}$.

In the other direction, $HK$ does not preserve resolutions, but there is still an isomorphism $I \to HKI$ with a transpose $I_!(I)\simeq \id_{\cat C\exlex} \to HK$ to serve as unit of the adjunction.
These natural transformations are the co-unit and unit of and adjunction $K\dashv H$. The functor $H$ is fully faithful because the co-unit of the adjunction is an isomorphism.%controleer! nee dat is niet leuk...
\end{proof}

The following equivalence wraps this subsection up.

\begin{theorem} Let $\cat D$ be exact and let $F:\cat C\to\cat D$ be a finitely complete embedding. The following are equivalent:
\begin{enumerate}
\item the embedding $F$ is resolvent;
\item the canonical functor $K:\cat C\exlex \to \cat D$ has a right adjoint right inverse;
\item for each exact $\cat D$, every finitely continuous functor $G:\cat C \to \cat D$ has a left Kan extension $F_!(G)$ which satisfies $F_!(G)F\simeq G$.
\end{enumerate}\label{ThA}
\end{theorem}

\begin{proof} Lemma \ref{LKE} gives $1\to 3$. Lemma \ref{converse} gives $3\to 1$. Lemma \ref{adjunct} gives $1$ and $3\to 2$. This leaves $2\to 1$.

If $H$ is a right adjoint right inverse of the canonical functor $K:\cat C\exlex \to \cat C\exlex$, consider for each object $X$ of $\cat D$ a resolution $c_X:IX_0 \to HX$. The resolution $c_X$ has a transpose $c_X^t:KIX_0 \simeq FX_0 \to X$. Each $f:FY\simeq KIX \to X$ in $\cat D$ has a transpose $f^t: IX \to HX$, which factors through $c_X$. Hence $f$ factors through $c_X^t$. By generalization every object has a resolution and $F$ is resolvent.
\end{proof}

The following section considers the special case that $F$ is the ex/reg completion $J:\cat C\to\cat C\exreg$ and that $\cat C\exreg$ is a topos.

\section{Ex/reg toposes}
The ex/reg completion $\cat C\exreg$ of a locally Cartesian closed regular category $\cat C$ is a topos if the category has a generic monomorphism, i.e.\ a monomorphism of which every monomorphism is a pullback \citep{Menni00exactcompletions}. This section generalizes this result to regular categories that are only weakly Cartesian closed and proves that the generic monomorphism implies that $J:\cat C\to\cat C\exreg$ is resolvent. This leads to another equivalence (theorem \ref{ThB}): $\cat C\exreg$ is a topos and $\cat C \to \cat C\exreg$ is resolvent in and only if $\cat C$ has weak dependent products and a generic monomorphism.

\subsection{Weak dependent products}
This subsection shows that when $J:\cat C \to \cat C\exreg$ is resolvent, $\cat C\exreg$ is locally Cartesian closed if and only if $\cat C$ has weak dependent products.

\newcommand\weakprod{\widetilde{\Pi}}
\begin{defin} Let $\cat C$ be an arbitrary finitely complete category. For each $f:X\to Y$ in $\cat C$, pullbacks along $f$ determine the \emph{reindexing functor} $f^*:\cat C/Y\to\cat C/X$ up to unique isomorphism. If this functor has a right adjoint $\prod_f:\cat C/X\to\cat C/Y$ for each $f$, then $\cat C$ is \emph{locally Cartesian closed}. 

For each $A$ of $\cat C/X$, the object $\prod_f(A)$ of $\cat C/Y$ is the \emph{dependent product} of $A$ along $f$. The co-unit $\epsilon_A: f^*(\prod_f(A))\to A$ has the property that for each $g:f^*(B) \to A$ there is a unique $g^t:B\to \prod_f(A)$ such that $\epsilon_A\circ f^*(g^t) = g$.

A \emph{weak dependent product} of $A$ along $f$ is an object $\weakprod_f(A)$ together with a morphism $e:f^*(\weakprod_f(A))\to A$ such that for each morphism $g:f^*(B)\to A$, there are $g':B\to \weakprod_f(A)$ such that $e\circ f^*(g') = g$. When $\cat C$ has weak dependent products for each $A$ along each $f$, $\cat C$ is \emph{weakly locally Cartesian closed}.
\end{defin}

The exact completions add quotients, which helps to construct dependent products from weakly dependent products. Let $\cat C$ be a finitely complete category.

\begin{lemma} The category $\cat C$ is weakly locally Cartesian closed if and only if $\cat C\exlex$ is locally Cartesian closed. \end{lemma}

\begin{proof} See \citet{MR1787592}. \end{proof}

\begin{prop} 
If $\cat C$ is regular and $J:\cat C\to\cat C\exreg$ is resolvent, then $\cat C$ is weakly locally Cartesian closed if and only if $\cat C\exreg$ is locally Cartesian closed. \label{lccc}
\end{prop}

\begin{proof} Let $\cat C$ be regular and let $J:\cat C\to\cat C\exreg$ be resolvent. Because $J$ is resolvent, $\cat C\exreg$ is a reflective subcategory of $\cat C\exlex$. If $\cat C$ is weakly locally Cartesian closed, then $\cat C\exreg$ inherits local Cartesian closure from $\cat C\exlex$. Next let $\cat C\exreg$ be locally Cartesian closed. For $f:X\to Y$ and $A\in \cat C/X$, there is a dependent product $\prod_{Jf}(JA)$ of $JA$ in $\cat C\exreg/JX$ along $Jf$. This dependent product has a resolution $h:J(\weakprod_f(A)) \to \prod_{Jf}(JA)$. Also, there is an $e:f^*(\weakprod_f(A)) \to A$, because $\epsilon_{JA}:(Jf)^*(\prod_{Jf}(JA)) \to JA$ composed with $Jf^*h$ is in the image of the fully faithful $J$. These form a weak dependent product. \end{proof}

\newcommand\N{\mathbb N}

\begin{example} The category of families $\Fam(\cat C)$ over any small finitely complete category $\cat C$--see example \ref{famex} for the definition--is weakly locally Cartesian closed, because $\Fam(\cat C)\exlex$ is equivalent to the topos of presheaves $\Set^{\cat C\dual}$. Let $\Set^\to$ be the arrow category of the category of sets, and let $\cat C$ be the small subcategory of $\Set^\to$ whose objects are functions between subsets of the set of natural numbers $\N$.  The exponential of $!:2\to 1$ and $\id_\N:\N\to\N$ (in $\Set^\to$) is $!:2^\N \to 1$, which is neither countable, nor a coproduct of countable morphisms. It therefore cannot exist in $\Fam(\cat C)$, so $\Fam(\cat C)$ is not Cartesian closed, let alone locally Cartesian closed. Note that $\Fam(\cat C)$ is regular and that $\Fam(\cat C)\exreg$ is a topos too. \end{example}%juiste plaats?

\subsection{Open monomorphisms}
For an object of all truth values in $\cat C\exlex$, $\cat C$ needs a generic proof \citep{MR1948025}. This paper shows what happens if we weaken this condition.

\begin{defin} A \emph{generic monomorphism} is a monomorphism of which every monomorphism is a pullback. \end{defin}

\begin{example} Like every topos, $\Set^\to$ has a generic monomorphism, but its ex/lex completion of $\Set^\to$ has not \citep[lemma 2.4]{MR2320014}. \end{example}

%verderop wordt epi-mono en locaal iso-prone ook gebruikt.
The subsection shows that a generic monomorphism in $\cat C$ induces a classifier of a class of monomorphism in $\cat C\exlex$, which I define with the help of the following concept.

\begin{defin}[Orthogonality] In any category $\cat C$, a morphism $e:A\to B$ is \emph{left orthogonal} to $m:C\to D$ and $m$ is \emph{right orthogonal to } $e$, if for all $f:A\to C$ and $g:B\to D$ such that $m\circ f = g\circ e$ there is a unique $h:B\to C$ such that $h\circ e = f$ and $m\circ h = g$.
\[ \xymatrix{
A \ar[r]^f\ar[d]_e & C \ar[d]^{m} \\
B \ar[r]_g \ar@{.>}[ur]|h & D
} \]
\end{defin}

If $\cat C$ is finitely complete and if $m:C\to D$ in $\cat C$ is left orthogonal to $e$, then so is the pullback of $m$ along any morphism $E\to D$. An important example of orthogonality is that regular epimorphisms are left orthogonal to all monomorphisms.

\begin{defin} A \emph{open monomorphism} in $\cat C\exlex$ is a monomorphism $m:V\to W$  which is right orthogonal to $Ie$ for every regular epimorphism $e:X\to Y$ in $\cat C$. \end{defin}

The following characterization lemma shows that open monomorphisms locally are monomorphisms in $\cat C$.

\begin{lemma} If $\cat C$ is regular, then a monomorphism $m:X\to Y$ is open if and only if pullbacks along regular epimorphisms $e:IY_0 \to Y$ are isomorphic to monomorphisms $Im_0:IX_0 \to IY_0$ in the image of $I$. \label{charclosed}\end{lemma}

\newcommand\pb{\ar@{}[dr]|<\lrcorner}

\begin{proof} First assume that $m:X\to Y$ is open. Let $m':X'\to IY_0$ be the pullback of $m$ along $p_Y:IY_0 \to Y$ and let $e':IX'_0 \to X'$ be a cover of $X'$. There is an epimorphism $e_0: X'_0 \to X_0$ and a monomorphism $m_0:X_0 \to Y_0$ such that $Im_0\circ Ie_0 = m'\circ e'$. Since pullback of $m$, $m'$ is open, so $m'$ is right orthogonal to $Ie_0$. Simultaneously, $Im_0$ is right orthogonal to $e'$ because $Im_0$ is a monomorphism and $e'$ is an epimorphism. For these reasons, there is a unique isomorphism $IX_0 \to X$ which commutes with $m$ and $Im_0$.

\[\xymatrix{
IX'_0 \ar[r]^{e'}\ar[d]_{Ie_0} & X'\ar[r]\ar[d]_{m'}\pb & X \ar[d]^m \\
IX_0 \ar[r]_{Im_0} \ar@{.>}[ur]|\simeq & IY_0\ar[r]_{p_Y} & Y
}\]

Next let $m:X\to Y$ be arbitrary, assume that $Im_0: IX_0 \to IY_0$ is the pullback of $m$ along some cover $p_Y:IY_0 \to Y$ and let $p_X:IX_0 \to X$ be the pullback of $p_Y$ along $m$. Let $e:X'\to Y'$ be a regular epimorphism in $\cat C$ and let $x:IX'\to X$, $y:IY'\to Y$ be morphisms that satisfy $m\circ x = y\circ Ie'$. Since $IY'$ is projective, $e$ is a regular epimorphism and $I$ is full, there are $y_0:Y'\to Y_0$ such that $p_Y\circ Iy_0 = y$. Since $Im_0$ is the pullback of $m$ along $p_Y:IY_0\to Y$ and $I$ is fully faithful, there is a unique $x_0:X'\to X_0$ such that $p_X\circ Ix_0 = x$ and $m_0\circ x_0 = y_0 \circ e$. But $m_0$ is right orthogonal to $e_0$ which means that there is a unique $z_0:Y'\to X_0$ such that $z_0\circ e = x_0$ and $m_0\circ z_0 = y_0$. Hence $z = p_X\circ Iz_0$ satisfies $z\circ Ie = x$ and $m\circ z = y$.

\[\xymatrix{
IX'\ar[d]_{Ie} \ar@{.>}[r]_{Ix_0}\ar@/^/[rr]^x & IX_0 \ar[d]^{Im_0} \ar[r]_{p_X} \pb & X \ar[d]^m\\ 
IY'\ar@{.>}[r]^{Iy_0} \ar@/_/[rr]_{y} \ar@{.>}[ur]|{Iz_0} & IY_0\ar[r]^{p_Y} & Y
}\]

For each $z': IY'\to Y$ such that $z'\circ Ie = x$ and $m\circ z' = y$, $z'\circ Iy_0$ factors through the same $Iz_0$, so $z' = z$ and $m$ is right orthogonal to $Ie$. By generalization each monomorphism $m:X\to Y$ is open if and only if its pullback along a cover $p_Y:IY_0 \to Y$ is isomorphic to a monomorphism in the image of $I$.
\end{proof}

The next lemma connects open monomorphisms with codomain $Z$ to open morphisms whose codomains cover $Z$.

\begin{lemma}[open descent] Let $e:Y\to Z$ be a regular epimorphism and let $p,q:X\to Y$ be a kernel pair of $e$. If $m:V\to Y$ is an open monomorphism whose pullbacks $p^*(m)$, $q^*(m)$ along $p$ and $q$ are isomorphic, then there is an up to isomorphism unique open monomorphism $n:W\to Z$ such that $m = e^*(n)$.
\[ \xymatrix{
 V \ar@{.>}[r]\ar[d]_m\pb & W\ar@{.>}[d]^n\\
 Y \ar[r]_e & Z
}\]\label{opendescent}
\end{lemma}

\begin{proof} The composite $e\circ m:$ factors as a monomorphism $n:W\to Z$ following a regular epimorphism $e':V\to W$ in an up to isomorphism unique way. Let $l:U\to X$ be a pullback of $m$ along $q$. Because $l$ is also a pullback of $m$ along $p$ there is a regular epimorphism $p':U\to V$ such that $p\circ l$ factors as $m\circ p'$. The composite $p\circ l$ is the pullback of $e\circ m = n\circ e'$, but because regular epimorphisms an monomorphisms are both stable under pullback, $p'$ is a pullback of $e'$ and $m$ is a pullback of $n$ as required. 

To show that $n$ is open, let $c_Z:IZ_0 \to Z$ be any regular epimorphism. Since $IZ_0$ is projective, there is a $d:IZ_0 \to Y$ such that $e\circ d = c_Z$. A pullback of $n$ along $c_Z$ is a pullback of $m$ along $d$, then therefore open. By lemma \ref{charclosed}, $n$ is an open monomorphism.
\end{proof}

\hide{Vanaf hier komen subobjecten e.d. weer terug. Verderop gebruik ik lokale kleinheid. Daar komen de problemen vandaan.

Misschien kunnen we onderstaande definitie helemaal vermijden.

Herformuleer wat we willen bewijzen? Er is een subobject classfier: generiek monomorfisme waarvoor de pullbakcs uniek zijn.
}

%--- hier iets over sub ?

\newcommand\exeq{\equiv}
\hide{%--- oud ---
\newcommand\cl{\mathcal O}
\begin{defin} For each object $X$ of $\cat C\exlex$ a \emph{open subobject} is an isomorphism class of open monomorphisms to $X$. Let $\cl(X)$ be the \emph{set} of open subobjects. Note that $\cl$ is a functor $\cat C\exlex\dual \to \Set$, because open monomorphisms are stable under pullback. \end{defin}

Let $\cat C$ be a regular category with a generic monomorphism $\gamma:E\to P$. I show next that $\cat C\exlex$ has an object $\Omega$ that represents $\cl$: the .
\begin{defin} Let $\gamma:E\to P$ be a generic monomorphism in $\cat C$. Because $\cat C\exlex$ is locally Cartesian closed, the inverse image maps $f\inv:\sub(Y)\to \sub(x)$ of each $f:X\to Y$ in $\cat C\exlex$ has a right adjoint $\forall_f$ which satisfies the Beck-Chevalley condition. Thank to these right adjoints, the following congruence on $IP$ exists in $\cat C\exlex$.
\[ (p\exeq q) \iff (p\in \db{I\gamma} \leftrightarrow q\in \db{I\gamma}) \]
Here $\db{I\gamma}$ is the subobject containing all monics isomorphic to $I\gamma$. Let $\Omega = IP/\mathord\exeq$ and let $c_\Omega$ be the quotient map $IP \to \Omega$.
\end{defin}
}

%--- nieuw ---

Let $\cat C$ be a regular category with a generic monomorphism $\gamma:E\to P$. The rest of this subsection shows that $\cat C\exlex$ has an \emph{open-subobject classifier}, i.e. a generic monomorphism $t:1\to\Omega$ such that every open monomorphism $X\to Y$ is the pullback of $t$ along a unique morphism $Y\to \Omega$.

\begin{defin} Let $\gamma:E\to P$ be a generic monomorphism in $\cat C$. Because $\cat C\exlex$ is locally Cartesian closed and exact, there is a congruence $\mathord\exeq\subseteq IP\times IP$ such that $(f,g):X\to IP\times IP$ factors through $\exeq$ if and only if pullbacks $f^*(I\gamma)$ of $I\gamma$ along $f$ are pullbacks of $I\gamma$ along $g$. The congruence $\exeq$ is the groupoid of fibers of $I\gamma$ and their isomorphisms. Let $\Omega = IP/\mathord\exeq$ and let $c_\Omega$ be the quotient map $IP \to \Omega$.
\end{defin}

\hide{
%--- oud ---
\begin{lemma} There is a natural isomorphism $\cat C\exlex (I-,\Omega) \to \cl(I-)$.\label{iso1} \end{lemma}

\begin{proof} The relation $\exeq$ is defined in such a way, that $c_\Omega\circ If = c_\Omega\circ Ig$ if and only if $f\inv(\db{I \gamma}) = g\inv(\db{I \gamma})$ for each pair $f,g:X\to P$. Meanwhile, $I:\sub(X) \to \cl(IX)$ is an isomorphism by lemma \ref{charclosed}. 
The map $\cat C\exlex (IX,\Omega) \to \cl(IX)$ takes a morphism $f:IX\to \Omega$, factors it as $c_\Omega\circ If_0$ and then pulls $I\gamma$ back along $If_0$. This is a well defined and injective map, because every factorization of $f$ through $c_\Omega$ gives the same open subobject of $IX$. It is surjective because $\gamma$ is a generic monomorphism.

If $g:X\to Y$, then $\cl(Ig):\cl(IY) \to \cl(IX)$ commutes with $\cat C\exlex(Ig,\Omega)$ because if $h = c_\Omega\circ Ih_0:Y\to \Omega$ then $h\circ g = c_\Omega\circ Ih_0\circ g$. 
\end{proof}
}

%--- nieuw ---
\begin{lemma} There is an open-subobject classifier $t:1\to \Omega$. \label{iso3} \end{lemma}

\begin{proof} The object $E$ in $\cat C$ has a global section because $\id_1:1\to 1$ is a monomorphism. Meanwhile the congruence $\exeq$ has the property that $IE\times IE$ is the pullback of $I\gamma$ and either projection $\pi_0:\mathord \exeq \to IP$. The composition $c_\Omega\circ I\gamma$ factors as a monomorphism $t:S\to \Omega$ following a regular epimorphism $!:IE\to S$. This $S$ is necessarily the support of $IE$ and therefore a terminal object.

For $X$ in $\cat C$ every open morphism $m:W\to IX$ is the pullback of $t:1\to \Omega$ along a unique morphism $IX\to \Omega$. The reason is that there are morphisms $\chi:X\to P$ such that $m \simeq \chi^*(\gamma)$ and that if $\chi^*(\gamma)\simeq \psi^*(\gamma)$, then $c_\Omega\circ \chi = c_\Omega\circ \psi$. For other $X$ in $\cat C\exreg$ every open morphism $m:W\to X$ is the pullback of $t:1\to \Omega$ along a unique morphism $X\to \Omega$, because $X$ has a cover $c_X:IX_0 \to X$ and the characteristic function of the pullback of $m$ along $c_X$ factors through $c_X$.\end{proof}

The following theorem generalizes the previous lemma and turns it into an equivalence.

\begin{theorem} If $\cat C$ is regular, then $\cat C\exlex$ has a classifier of open subobjects if and only if $\cat C$ has a generic monomorphism. \label{closubclass}\end{theorem}

\begin{proof} If $\cat C$ has a generic monomorphism then apply lemma \ref{iso3}.

Assume that $\cat C\exlex$ has an open-subobject classifier $t:1\to \Omega$. Of course $\Omega$ has a cover $c_\Omega:IP\to \Omega$ and there is a monomorphism $\gamma:E\to P$ such that $I\gamma$ is the pullback of $t:1\to \Omega$ along $q$. For every monomorphism $m:X\to Y$ in $\cat C$, there is an $f:IY \to \Omega$ such that $Im$ is the pullback of $t$ along $k$. Because $IY$ is projective and because $I$ is fully faithful, there is a $g:Y \to P$ such that $q\circ Ig = f$. Hence $m$ is the pullback of $\gamma$ along $g$. By generalization $\gamma$ is a generic monomorphism.
\end{proof}

\subsection{Monadicity}
In a topos $\cat E$ the power object functor $\pow:\cat E\to \cat E$ is monadic. If $\cat C$ is regular, has weak dependent products and is a generic monomorphism, then the open-subobject classifier $t:1\to\Omega$ in $\cat C\exlex$ induces something like a power object functor $\Omega^-: \cat C\exlex\dual \to \cat C\exlex$. This subsection singles out a full subcategory $\cat S$ of $\cat C\exreg$, such that $\Omega^-$ is a monadic functor $\cat S\dual \to\cat S$. Thanks to other properties of $\Omega$, this $\cat S$ is a topos.

\hide{Een voorbeeld van een reguliere cat met zwakke afhankelijke producten en een generiek monomorfisme die geen sterke producten heeft. Dit lijkt me lastiger.

Moeten we daar hier nog even naar verwijzen?
}

\newcommand\clopow{\Omega^}
\newcommand\unit{\eta_}
\begin{defin} For each object $A$, let $\unit A:A\to \clopow{\clopow A}$ be the unit of the adjunction $\clopow- \dashv \clopow-$. This transformation has the property that $\clopow{\unit A}\circ \unit{\clopow A} = \id_{\clopow A}$. In order to stop towers of $\Omega$, let $TA = \clopow{\clopow A}$.
An object $X$ of $\cat C\exlex$ is \emph{Stone} if $\unit X$ is the equalizer of $\unit{TX}$ and $T\unit X$. The category $\cat S$ is the full subcategory of Stone objects in $\cat C\exlex$.
\end{defin}

I use the name `Stone' because we can recover these objects from their objects of open subobjects like we can recover \emph{Stone spaces} from their lattices of open subspaces.
In order to proof a monadicity for $\Omega^-$ and Stone objects, I use a couple of properties of open monomorphisms and their classifier, which I haven't mentioned yet.

\begin{defin} An object $X$ in $\cat C\exlex$ is \emph{discrete} if the diagonal $\delta:X\to X\times X$ is an open monomorphism. \end{defin}

\begin{remark} The category of discrete objects in $\cat C\exlex$ is a \emph{q-topos} as in \citet{a2CAotTtTC}. There is a category of \emph{coarse objects} among the discrete, which is a topos. These coarse objects are precisely sheaves for the canonical topology of $\cat C$ in the terminology of \citet{MR1870615}. The coming subsections determine the full subcategory of these objects and prove that it is a topos. \end{remark}

\begin{lemma} Open monomorphisms and their classifier have the following useful properties.
\begin{enumerate}
\item For each object $X$ and each open monomorphism $m:Y\to Z$, $m^X:Y^X \to Z^X$ is open.
\item For each object $X$, $\Omega^X$ is discrete.
\end{enumerate}\label{compactdiscrete}
\end{lemma}

\begin{proof} For (1) let $m:Y\to Z$ be an open monomorphism, let $X$ be any object and let $c_X:IX_0 \to X$ be a regular epimorphism. Let $e:V\to W$ be a regular epimorphism of $\cat C$. The morphism $m$ is right orthogonal to $Ie\times \id_X$ because it is right orthogonal to both $Ie\times \id_{IX_0}$ for being open, and to $\id_{W}\times c_X$ for being a monomorphism. Hence $m^X$ is right orthogonal to $Ie$. By generalization $m^X$ is open whenever $m$ is.

\[\xymatrix{
IV\times IX_0 \ar[d]_{Ie\times \id} \ar[r]^{\id\times c_X} & IV\times X \ar[r] \ar[d]|{Ie\times \id} & Y\ar[d]^m && IV \ar[d]_e\ar[r] & Y^X \ar[d]^{m^X} \\
IW\times IX_0 \ar[r]_{\id\times c_X}\ar@{.>}@/^/[urr] & IW\times X \ar[r]\ar@{.>}[ur] & Z && IW\ar[r]\ar@{.>}[ur] & Z^X
}\]

For (2) let $e:X\to Y$ be a regular epimorphism in $\cat C$ and let $f:IX\to \Omega$ and $(g,h):IY\to \Omega\times \Omega$ be morphisms that satisfy $Ie\circ f = g = h$. Since $\Omega$ is the open-subobject classifier and open subobjects of $IX$ and $IY$ come from ordinary subobjects of $X$ and $Y$ in $\cat C$, $g$ and $h$ correspond to $U_g,U_h\subset Y$ such that $e^{-1}(U_g) = e^{-1}(U_h)$. Since the inverse image map of a regular epimorphism is injective, $U_g=U_h$. Therefore $g=h$. This establishes that $\delta:\Omega\to \Omega\times \Omega$ is right orthogonal to $Ie$. By generalization $\delta$ is open. By part 1, $\delta^X: \Omega^X \to \Omega^X\times \Omega^X$ is open too.
\end{proof}

The category $\cat S$ is a topos because $\clopow-$ is a monadic functor $\cat S\dual\to\cat S$. The following lemma shows that $\clopow-$ indeed lands in $\cat S$. Therefore $\clopow-$ is is right adjoint to $(\clopow-)\dual$. The lemma after the next that proves that $\clopow-$ satisfies another condition of Beck's monadicity theorem \citep[VI.7]{MR0354798}.

%For a monadicity theorem I need two things. The functor $\Omega^-:\cat S\dual \to \cat C\exlex$ needs a left adjoint. If $\Omega^-$ factors through the inclusion $\cat S \to \cat C\exlex$, then $\Omega^-$ is adjoint to $(\Omega^-)\dual$. The functor $\Omega^-:\cat S\dual \to \cat S$ creates split coequalizers. Let's prove that.

\begin{lemma} For each object $X$ of $\cat C\exlex$, $\Omega^X$ is Stone.\end{lemma}

\begin{proof} The morphisms $\unit{\clopow{X}}:\clopow X\to T\clopow X$, $\unit{T \clopow X}$ and $T\unit{\clopow X}:T\clopow X \to TT\clopow X$ are a split equalizer for the following reasons. The maps $\clopow{\unit X}$ and $\clopow{\unit{TX}}$ satisfy $\clopow{\unit X}\circ \unit{\clopow X} = \id_{\clopow X}$ and $T\clopow{\unit{X}}\circ T\unit{\clopow{X}} = \id_{T\clopow X}$. Naturalness makes $\unit{T \clopow X}\circ \unit{\clopow{X}} = T\unit{\clopow X}\circ\unit{\clopow{X}}$ and $\unit{\clopow X}\circ \clopow{\unit X} = T\clopow{\unit X}\circ \unit{T\clopow X}$. Split equalizers are absolute limits and therefore $\clopow-$ sends them to split coequalizers. Hence $\clopow X$ is Stone.
\end{proof}

%check
\begin{lemma}[Creating split coequalizers] Let $h:W\to X$ and $f,g:X\to Y$ be in $\cat S$. If $\Omega^f$ and $\Omega^g$ has a split coequalizer and $f\circ h = g\circ h$, then $\Omega^h$ is a split coequalizer of $\Omega^f$ and $\Omega^g$ if and only if $h$ is the equalizer of $f$ and $g$. \end{lemma}

\begin{proof} \emph{Firstly the `only if' direction:} $Y$ is a discrete object because it is a subobject of $TY=\Omega^{\Omega^Y}$. If $h$ is an equalizer of $f$ and $g$, then $h$ is open and therefore composition with $h$ on the left is a section of $\Omega^h$. So $\Omega^h$ is a split epimorphism, and $\Omega^f\circ \Omega^h=\Omega^g\circ\Omega^h$. It is easy to see that if $\Omega^f$ and $\Omega^g$ have a split coequalizer, then it must equal $\Omega^h$.%echt waar!

\emph{Secondly the `if' direction:} $\eta_W$ already is an equalizer of $T\eta_W$ and $\eta_{TW}: TW \to TTW$. The functor $\Omega^-$ maps split coequalizers to split equalizers, so if $\Omega^h$ is a split coequalizer of $\Omega^f$ and $\Omega^g$, then $Th$ is a split equalizer of $Tf$ and $Tg$, i.e.\ there are morphisms $h':TX \to TW$ and $f':TY\to TX$ such that $h'\circ Th = \id_{TW}$, $f'\circ Tf = \id_{TY}$ and $f'\circ Tg = Th\circ h'$.

Suppose $k: V\to X$ satisfies $f\circ k$ and $g\circ k$. Let $k'= h'\circ \eta_X\circ k$. The following diagram shows that $Th\circ k' = \eta_X\circ k$.
\hide{ Th h' \eta_X k = f' Tg \eta_X k = f' \eta_Y  g k = f' \eta_Y  f k = f' \eta_Y  f k = \eta_X\circ k}
\[ \xymatrix{
V\ar[r]^k\ar[d]_k & X\ar[r]^{\eta_X}\ar[d]|g & TX\ar[r]^{h'}\ar[d]|{Tg} & TW \ar[d]|{Th} \\
X \ar[r]|f \ar[dr]_{\eta_X} & Y \ar[r]|{\eta_Y} & TY \ar[r]|{f'} & TX \\
& TX\ar[ur]|{Tf} \ar[urr]_{\id}
}\]

Using this equality the following diagram shows that $T\eta_W \circ k' = \eta_{TW}\circ k'$.
\hide{
\underline{\id_{TTW}} T\eta_W  h' \eta_X  k = 
Th' \underline{TTh T\eta_W}  h' \eta_X  k =
Th' T\eta_X  \underline{Th  h'} \eta_X  k =
Th' T\eta_X  f' \underline{Tg  \eta_X}  k =
Th' T\eta_X  f' \eta_Y  \underline{g  k} =
Th' T\eta_X  \underline{ f' \eta_Y  f}  k =
Th' \underline{T\eta_X  \eta_X}  k =
\underline{Th' \eta_{TX}} \eta_X  k =
\eta_{TW} h' \eta_X  k
}
\[ \xymatrix{
& X \ar[r]^{\eta_X}\ar[d]|{\eta_X} & TX \ar[r]^{h'}\ar[d]|{\eta_{TX}} & TW \ar[d]^{\eta_{TW}}\\
V\ar[dr]_{k'}\ar[ur]^k & TX \ar[r]|{T\eta_X} & TTX \ar[r]|{Th'} & TTW \\
& TW\ar[r]_{T\eta_W}\ar[u]|{Th} & TTW \ar[ur]_\id \ar[u]|{TTh}
}\]
Because $\eta_W$ is the equalizer of $T\eta_W$ and $\eta_{TW}$, $h'\circ \eta_X \circ k = \eta_W\circ \alpha$ for a unique $\alpha: W\to V$. It follows that $\eta_X \circ k = \eta_X \circ h\circ \alpha$ as the following diagram shows.
\[ \xymatrix{
& X \ar[dr]^{\eta_X} \\
V \ar[r]|{k'}\ar[ur]^k\ar[dr]_\alpha &  TW\ar[r]|{Th} & TX \\
& W \ar[r]_h \ar[u]|{\eta_W} & X \ar[u]_{\eta_X}
}\]
Because $X$ is Stone $\eta_X$ is a monomorphism. Hence $k = h\circ \alpha$. Since each $k$ that satisfies $f\circ k = g\circ k$ factors uniquely through $h$ and $f\circ h = g\circ h$, $h$ must be the equalizer of $f$ and $g$.
\end{proof}

\begin{theorem}[Monadicity] The functor $\clopow-:\cat S\to\cat S$ is monadic. Therefore $\cat S$ is a topos. \end{theorem}

\begin{proof} Monadicity follows from Beck's monadicity theorem \citep[VI.7]{MR0354798}.

The definition of open en discrete here and in \citet{MR1799865} are the same. An object $X$ of $\cat S$ is \emph{compact} if $t^X:1^X\to \Omega^X$ is open, and lemma \ref{compactdiscrete} proves both. Theorem 11.18 in \citet{MR1799865} shows that $\cat S$ is a topos.\end{proof}

A practical description of Stone objects concludes this subsection.

\begin{lemma} An object is Stone if and only if it is an open subobject of $\Omega^X$ for some $X$ of $\cat C$. \label{Stoneequiv} \end{lemma}

\begin{proof} If $S$ is Stone then $\unit S$ is open, because $\Omega^X$ are discrete for all $X$.
The object $1$ is Stone because $!:T1 \to 1$ is inverse to $\unit 1: 1\to T1$ and that makes $\unit 1$ a split equalizer of $T\unit 1$ and $\unit{T1}$.
Since $\cat S$ is a full subcategory and a topos, it contains $t:1\to \Omega$ and all of its pullbacks.
\end{proof}

The full subcategory $\cat S$ of Stone objects in $\cat C\exlex$ is a topos. The next subsections define a full subcategory $\cat L$ of \emph{local objects} and prove that it is equivalent to $\cat C\exreg$ and equal to $\cat S$.

\subsection{Local objects}
Split epimorphisms are the only regular epimorphisms that $I:\cat C \to\cat C\exlex$ preserves. Some objects of $\cat C\exlex$ still treat regular epimorphisms from $\cat C$ as regular epimorphisms. More precisely, if $\cat C\exlex$ is locally small then the functor $\cat C\exlex(I-,X):\cat C\dual \to \Set$ is a sheaf for the regular topology on $\cat C$. This subsection characterizes the objects with this property.

\begin{defin}[Local isomorphisms] A monomorphism $d:X\to IY$ is a \emph{local isomorphism}, if there is a regular epimorphism $d_0:X_0 \to Y$ in $\cat C$ such that $Id_0$ factors through $d$.
\end{defin}

\begin{lemma}[Properties of local isomorphisms] Local isomorphisms are stable under pullback (along morphisms in $\cat C$) and left orthogonal to open monomorphisms. A monomorphism $d:X\to IY$ is a local isomorphism if and only if there is a regular epimorphism $e:Z\to Y$ such that the pullback of $d$ along $Ie$ is an isomorphism. \end{lemma}

\begin{proof} Let $d:X\to IY$ be a local isomorphism. In particular let $d_0:X_0 \to Y$ be a regular epimorphism in $\cat C$ and let $c:IX_0 \to X$ satisfy $d\circ c= Id_0$.

Let $h: W\to Y$ be any morphism. A pullback $h^*(d_0)$ of $d_0$ along $h$ is a regular epimorphism and $Ih^*(d_0)$ factors through any pullback $Ih^*(d)$ of $d$ along $Ih$, making $Ih^*(d)$ a local isomorphism.
\[ \xymatrix{
\bullet\ar[d]\ar[r]^{Ih^*(c)} \pb & \bullet \ar[r]^{Ih^*(d)} \ar[d]\pb & IW\ar[d]^{Ih}\\
IX_0 \ar[r]_c & X \ar[r]_d & IY
}\]
By generalization all local isomorphisms are stable under pullback.

Let $m: W\to Z$ be an open monomorphism and suppose that $x: X\to W$ and $y: IY\to Z$ satisfy $y\circ d = m\circ x$. In that case $m \circ x\circ f = y\circ d\circ f = y\circ Id_0$ and there is a unique $z: IY_0 \to W$ such that $z\circ Id_0 = x\circ f$ and $m\circ z = y$, because $m$ is open. Because $Id_0 = d\circ f$ and $f$ is an epimorphism, $z\circ d = x\circ f$. Therefore $d$ is left orthogonal to $m$. By generalization local isomorphisms are left orthogonal to open monomorphisms.
\[ \xymatrix{
IX_0 \ar[r]^f \ar[dr]|(.35){Id_0}\ar[d]_f & X \ar[r]^x \ar[dl]|(.35)\id & W \ar[d]^m \\
X \ar[r]_{d}  & IY\ar[r]_y \ar@{.>}[ur]|z & Z
}\]

If the pullback $Ie^*(d)$ of $d:X\to IY$ along $Ie:X\to Y$ is an isomorphism, then $Ie$ factors through $d$, so it is a local isomorphism. On the other hand, if $d_0:X_0 \to Y$ is a regular epimorphism such that $Id_0$ factors through $d$, then the pullback $Id_0^*(d)$ of $d$ along $Id_0$ is a monic factor of the pullback $Id_0^*(Id_0)$ of $Id_0$ along itself. Because $Id_0^*(Id_0)$ is a split epimorphism, $d$ must be an isomorphism.
\end{proof}

The following class of objects has the property that they treat local isomorphisms as isomorphisms.

\begin{defin} An object $X$ of $\cat C\exlex$ is a \emph{local object}, if for each local isomorphism $m:Y\to IZ$ each $f:Y \to X$ factors uniquely through $m$. Let $\cat L$ be full subcategory of local objects in $\cat C\exlex$. \end{defin}

As the name suggests a local object $X$ has the property that $\cat C\exlex(I -,X)$ is a sheaf for the regular topology of $\cat C$ if $\cat C$ is locally small. Local objects are the `sheaves' of \citet{MR1870615}.

\begin{lemma} The functor $I:\cat C\to \cat C\exlex$ factors through $\cat L$ and $I:\cat C \to\cat L$ is a \emph{resolvent} regular functor.\label{iberegular} \end{lemma}

\begin{proof} Let's first consider the relation between regular epimorphisms in $\cat C$ and local objects in $\cat C\exlex$.
Let $e:X\to Y$ be a regular epimorphism of $\cat C$ and let $p,q:W\to X$ be a kernel pair in $e$. Note that $Ie = d\circ c$ where $d:X'\to IY$ is a local isomorphism and $c:IX\to Y$ a regular epimorphism. For each object $Z$ and each morphism $f:IX\to Z$, $f\circ Ip = f\circ Iq$ if and only if there is a unique $g:Y'\to Z$ such that $g\circ c= f$. If $Z$ is local, then $g$ factors uniquely through $d$.
\[\xymatrix{
IW \ar@<1ex>[r]^{Ip}\ar@<-1ex>[r]_{Iq} & IX \ar@/^/[rr]^{Ie}\ar[r]_{c}\ar[dr]_f & X'\ar[r]_d\ar[d]^g & IY\ar@{.>}[dl] \\
&& Z
}\]

On the other hand, if $Z$ has the property that for each coequalizer fork $p,q:W\to X$, $e:X\to Y$ in $\cat C$ each $f:IX\to Z$ such that $f\circ Ip=f\circ Iq$ factors through $IY$, then $Z$ is a local object. The reason is that if $d:X'\to IY$ is a local isomorphism and if $c:IX \to X'$ is a regular epimorphism, then $d\circ c = Ie$ for some regular epimorphism $e:X \to Y$, because $I$ is fully faithful and there is some regular epimorphism $d_0: X_0\to Y$ such that $Id_0$ factors through $Ie$. For each $g:X' \to IY$, let $f = g\circ c$. For any kernal pair $p,q:W\to X$ of $e$, $f\circ Ip = f\circ Iq$. Therefore $f$ factors uniquely through $e$ and $g$ factors uniquely through $d$.

For $X\in\cat C$, $IX$ is local because $I$ is fully faithful. Local objects are closed under finite limits because they are right orthogonal, so $I$ remains finitely continuous. The functor $I:\cat C\to\cat L$ it regular because it preserves coequalizers. The resolutions $c_X:IX_0 \to X$ are morphisms in $\cat L$, hence $I$ is still resolvent.
\end{proof}

The following corollary of this lemma shows that $\cat L$ is close to $\cat C\exreg$ even if $\cat C$ is an ordinary regular category.

\begin{corol} If $\cat L$ is exact, then $\cat L\simeq \cat C\exreg$. \label{LisCer}\end{corol}

\begin{proof} Let $F:\cat C\to\cat E$ be a regular functor to an exact category. There is an up to isomorphism unique finitely continuous functor $I_!(F):\cat L \to \cat E$ such that $I_!(F)I\simeq F$, because $I$ is resolvent. If $e:X\to Y$ is a regular epimorphism in $\cat L$, then for each resolution $c_Y:IY_0 \to Y$ there is a regular epimorphism $e_0:X_0 \to Y_0$ and a map $c_X: IX_0 \to X$ such that $e\circ c_X = c_Y \circ Ie_0$. This functor $I_!(F)$ sends resolutions to regular epimorphisms and $I_!(F)Ie = Fe$ is a regular epimorphism, so $I_!(F)e\circ I_!(F)c_X = I_!(F)c_Y \circ I_!(F)Ie_0$ is a regular epimorphism and that makes $I_!(F)e$ a regular epimorphism too. Hence $I_!(F)$ is a regular functor. If $\cat L$ is exact, then it satisfies the universal property of $\cat C\exreg$ and hence the categories are equivalent.
\end{proof}

The next lemma characterizes the members of $\cat L$.

\begin{lemma} An object $X$ of $\cat C\exlex$ is a local object if and only if it is an open subobject of $\Omega^Y$ for some $Y$ of $\cat C\exlex$. \label{localequiv}\end{lemma}

\begin{proof} \emph{Let's start with the `only if' direction.} If $X$ is a local object, then every monomorphism $m:X\to Y$ is open for the following reasons. Let $e:W\to Z$ be a regular epimorphism in $\cat C$, let $f:IW\to X$, $g:IZ\to Y$ satisfy $g\circ Ie = m\circ f$. Because $\cat C\exlex$ is regular $Ie = n\circ p$, with $n:V\to IZ$ a local isomorphism and $p:IW\to V$ a regular epimorphism. 
Since regular epimorphisms are left orthogonal to monomorphisms, there is a unique map $h:V\to X$ such that $h\circ p = f$ and $m\circ h = g\circ n$. Since $X$ is a local object, there is a unique $k:IZ\to X$ such that $k\circ n = h$ and $m\circ k = g$. Since this gives a unique factorization for every regular epimorphism $e$ of $\cat C$, $m$ is an open monomorphism.

\[ \xymatrix{
IW \ar[r]^p \ar[d]_{f}  & V \ar[r]^n\ar@{.>}[dl]|h & IZ \ar[d]_g\ar@{.>}[dll]|k \\
X \ar[rr]_m && Y
}\]

This property implies that the diagonal $X\to X\times X$ of every local object is open, which makes the singleton map $\set\cdot:X\to \Omega^X$ a monomorphism, which in turn is also open. Hence every local object is an open subobject of local object.

\emph{The `if' direction} starts with $\Omega$ itself. Suppose $d:X\to IY$ is a local isomorphism and that $u:X\to \Omega$ is some morphism. Let $m: U\to X$ be a pullback of $t:1\to \Omega$ along $u$. This is the pullback of some monomorphism $In:IV \to IY$ for the following reasons.

Let $c_X:IX_0 \to IY$ be a cover of $X_0$. The proof of lemma \ref{iberegular} shows that $d\circ c_X = Ib$ where $b$ is a regular epimorphism.
The pullback of $m$ along $c_X$ is some $Im_0$, where $m_0:U_0 \to X_0$ is a monomorphism in $\cat C$. Because $\cat C$ is regular, $b\circ m_0 = n\circ a_0$ for some monomorphism $n:V\to Y$ and some regular epimorphism $a_0:U_0\to V$. Let $c_U: IU_0\to U$ be the pullback of $c_X$ along $m$. Then $d\circ m\circ c_U = Ib\circ Im_0 = In\circ Ia_0$ and because $c_U$ is left orthogonal to $In$, there is a unique $a:U\to IV$ such that $In \circ a = d\circ m$--making $a$ a monomorphism--and $a\circ c_U = Ia_0$--making $a$ a local isomorphism.

\[\xymatrix{
IU_0 \ar[d]_{Im_0} \ar[r]_{c_U} \ar@/^/[rr]^{Ia_0}\ar@{}[dr]|<\lrcorner & U \ar[r]_a \ar[d]_{m} & IV \ar[d]^{In}\\
IX_0 \ar[r]^{c_X} \ar@/_/[rr]_{Ib} & X \ar[r]^d & IY
}\]

The monomorphism $In$ is open and $In \circ a = d\circ m$. The next part of this proof shows that $m$ is the pullback of $In$.

Let $W$ be an arbitrary object of $\cat C\exlex$ and let $p: W\to X$ and $q:W\to IV$ satisfy $d\circ p = In\circ q$. There is a cover $c_W:IW_0 \to W$. Let $a':Z\to IW_0$ be the pullback of $a$ along $q\circ c_W$ and let $q':Z\to U$ be the pullback of $q\circ c_W$ along $a$.

% d m q' = Im a q' = Im q c_W a' = d p c_W a'
The following diagram shows $d\circ m \circ q' = d\circ p \circ c_W \circ a'$ and since $d$ is a monomorphism, this means $m\circ q' = p \circ c_W \circ a'$. The local isomorphism $a'$ is left orthogonal to $m$, so there is a unique $h:IW_0 \to U$ such that $h\circ a'= q'$ and $m\circ h = p\circ c_W$. Being regular epimorphism, $c_W$ is left orthogonal to $m$ too, so there a unique $k: W\to U$ such that $m\circ k = p$ and $k\circ c_W = h$.
\[ \xymatrix{
Z \ar[rr]^{q'} \ar[d]_{a'} && U \ar[r]^m \ar[d]|a & X\ar[d]^d \\
IW_0 \ar[r]_{c_W}\ar@{.>}[urr]|h & W \ar[r]_q \ar[dr]_{p}\ar@{.>}[ur]|k & IV\ar[r]_{In} & IY \\ 
& & X \ar[ur]_d \ar@/^/[uur]^(.6)\id
}\]

% In a k = d m k = d p = In q
The following diagram show that $In\circ a\circ k = In\circ q$ and since $In$ is a monomorphism, this means $a k = q$.
\[ \xymatrix{
& U \ar[d]|m \ar[r]^a & IV \ar[d]^{In} \\  
W \ar[ur]^k \ar[r]_p \ar[dr]_q & X\ar[r]_d & IY \\
& IV \ar[ur]_{In} \ar@/^/[uur]^(.6)\id
}\]

Thus $(p,q)$ factors uniquely through $(m,a)$. By generalization the latter is a pullback cone for $(In, d)$. The open monomorphism $In$ is the pullback of $t:1\to \Omega$ along some $v:IV\circ \Omega$ and $v\circ d = u$, because the pullback of $In$ along $d$ is $m$. Therefore $\Omega$ is a local object.

Proving that other objects are local is easier now. For each $X$ of $\cat C$ the functor $-\times IX$ preserves local isomorphisms. Hence $-^{IX}$ preserves local objects.

Let $m:X\to Y$ be an open monomorphism and let $Y$ be local. Let $d: W\to IZ$ be a local isomorphism and let $x:W\to X$ be any morphism. Because $Y$ is local, there is a unique $y: IZ\to Y$ such that $y\circ d = m\circ x$. Since $d$ is left orthogonal to $m$, there is a unique $z:Z\to X$ such that $z\circ d = x$ and $m\circ d = y$. By generalization open subobjects of local object are local objects.
\[ \xymatrix{
W\ar[d]_d \ar[r]^x & X\ar[d]^m \\
IZ \ar@{.>}[r]_{y} \ar@{.>}[ur]|z & Y
}\]

Each object $X$ of $\cat C\exlex$ has a cover $c_X:IX_0\to X$ which is the coequalizer of some pseudo-equivalence relation $s,t:X_1 \to X_0$ in $\cat C$. The functor $\Omega^-$ sends coequalizers to equalizers, so $\clopow{c_X}:\clopow{X} \to \clopow{IX_0}$ is the equalizer of $\clopow s$ and $\clopow t$. Since $\clopow {X_1}$ is discrete, $\clopow{c_X}$ is an open monomorphism.

So $\Omega^X$ and every open subobject of $\Omega^X$ is local.
\end{proof}

\subsection{Equivalence}
The previous two subsections prove the following coincidence.

\begin{lemma} An object of $\cat C\exlex$ is local if and only if it is Stone. \end{lemma}

\begin{proof} Lemmas \ref{Stoneequiv} and \ref{localequiv} both establish an equivalence with the category of open subobjects of $\clopow{X}$ for arbitrary $X$. \end{proof}

This coincidence has the following consequence for the ex/reg completion.

\begin{lemma} The exact category $\cat C\exreg$ is a topos and $J:\cat C \to \cat C\exreg$ is resolvent. \end{lemma}

\begin{proof} Because $\cat L\cong \cat S$ and $\cat S$ is a topos, $\cat L$ is exact, and hence equivalent to $\cat C\exreg$, which is therefore also a topos by corollary \ref{LisCer}. The equivalence connects $J:\cat C \to\cat C\exreg$ to $I:\cat C \to \cat L$, which is resolvent by lemma \ref{iberegular}.

Let $X$ be any object of $\cat C\exreg$. The object $HX$ has a cover $c_X:IX_0 \to HX$ where $IX_0$ is projective. It follows that $Kc_X: KIX_0\simeq JX_0\to KH\simeq X$ is a resolution. 
\end{proof}

This section therefore concludes as follows.

\begin{theorem} For each regular $\cat C$ the following are equivalent:
\begin{itemize}
\item $\cat C$ is weakly locally Cartesian closed and has a generic monomorphism;
\item $\cat C\exreg$ is a topos and $J:\cat C \to\cat C\exreg$ is resolvent.
\end{itemize}\label{ThB}
\end{theorem}

\begin{proof} The previous lemma combined with theorem \ref{ThA}, proposition \ref{lccc} and theorem \ref{closubclass} prove this. \end{proof}

\hide{Menni's vraag over $\cat L$ is hiermee ook beantwoord: als $\cat L$ een topos is, dan is hij $\cat C\exreg$, en daar volgt verder alles uit. }

\section{Tripos theory}
This section demonstrates that resolvent functors play an important role in tripos theory. Toposes constructed from triposes are examples of ex/reg completions of regular categories with weak dependent products and a generic monomorphism thank to resolvent functors.

\subsection{Triposes}\label{TriThe}
Triposes are first order hyperdoctrines with some added structure. First order hyperdoctrines assign a `Heyting algebra of subobjects' to each object in a category of `types' and add some maps between the Heyting algebra for the purpose of interpreting quantification over those types. Triposes permit higher order quantification, because every type has an `intensional power type'.

\begin{defin} A \emph{tripos} is a presheaf $T$ on a category $\cat B$ that has finite products, with the following properties.
\begin{enumerate}
\item The presheaf $T$ is an internal Heyting algebra of the category of presheaves, i.e. for each $f:X\to Y$, $TX$ and $TY$ are Heyting algebras and $Tf:TY\to TX$ is a morphisms of Heyting algebras.
\item For each $f:X\to Y$ of $\cat B$, $Tf$ has both adjoints $\exists_f\dashv Tf \dashv \forall_f$ relative to the ordering of the Heyting algebras $TX$ and $TY$. 
These adjoints are stable under products, i.e.\ for all $f:A\to B$ and $g:C\to D$ in $\cat B$, 
\[ \exists_{f\times \id_C}\circ T(\id_A\times g) = T(\id_B\times g)\circ \exists_{f\times \id_D} \]
\item For each object $X$ of $\cat B$ there is an object $\pi X$ of $\cat B$ and a \emph{generic predicate} $\epsilon X\in T(\pi X\times X)$. Here, `$\epsilon_X$ is a generic predicate' means that for each object $Y$ and each $\alpha\in T(Y\times X)$ there is a map $a:Y\to \pi X$ such that $T(a\times \id_X)(\epsilon X) = \alpha$.
\end{enumerate}
\end{defin}

`Tripos' stands for \emph{topos representing indexed partially ordered set}. For each tripos $T:\cat B\dual \to \Set$, `$T$-valued partial equivalence relations' represent quotients of subobjects of triposes and `$T$-valued functional relations' represent morphisms between quotients of subobjects. Thanks to the presence of `weak power types' in the form of generic predicates, a tripos represents a topos.

\newcommand\inv{^{-1}}
\newcommand\im{\exists_}
\begin{defin} The set $T(X\times Y)$ is like a set of relations between $X$ and $Y$. In order to manipulate these relations, I introduce the following operations.
\begin{itemize}
\item Composition of $T$-valued relations is defined as follows. Let $e\in T(X\times Y)$ and $f\in T(Y\times Z)$, let $\pi_{02}(x,y,z) = (x,z)$, $\pi_{12}(z,x,y) = (y,z)$ and $\pi_{01}(x,y,z) = (x,y)$ be the projections of $X\times Y\times Z$.
\[ f\circ e = \im{\pi_{02}} T\pi_{12}(f)\land T\pi_{01}(e) \]
\item Inversion of $T$-valued relations is defined as follows. Let $e\in T(X\times Y)$ and let $\sigma(x,y) = (y,x)$. 
\[ e\inv = T\sigma(e) \]
\end{itemize}

The objects of $\cat B[T]$ are pairs $(X,e)$ where $e\in T(X\times X)$ is a \emph{partial equivalence relation}, i.e.\ a symmetric and transitive relation. Using the operators above, this means:
\begin{align*} e\inv &\leq e & e\circ e &\leq e \end{align*}
A morphism $(X,e) \to (Y,e')$ is a \emph{functional relation} $f\in T(X\times Y)$, which means $f\leq T((x,y)\mapsto (x,x))(e)$ and:
\begin{align*}
e'\circ  f = f \circ e &= f &
e &\leq f\inv \circ f &
f\circ f\inv &\leq e' 
\end{align*}
The $\circ$-operator for $T$-valued relations functions as composition of functional relations. The identity morphism of $(X,e)$ is just $e$ itself.
\end{defin}

\hide{ $f(x,y)  = \exists_z f(z,y)\land e(z,x)$ geeft $f(x,y) \leq e(x,x)$? Als het waar is, dan is het niet triviaal. }

\begin{lemma} For each object $X$ of $\cat B$, let $\nabla X = (X,\im{\delta_X}(\top))$ where $\delta_X:X\to X\times X$ is the diagonal; for each $f:X\to Y$ in $\cat B$, $\nabla f = \im{(\id_X,f)}(\top)$. This makes $\nabla$ a finite product preserving functor $\cat B \to \cat B[T]$.%niet eens full
\end{lemma}

\begin{proof} See \citet[section 2.4]{MR2479466}. \end{proof}

\newcommand\sub{\mathsf{Sub}}
\subsection{Resolvency}
This subsection characterizes functors $F:\cat B\to\cat E$ which are equivalent to $\nabla: \cat B \to \cat B[T]$ if $\cat E$ is well powered. The resolvency of a functor dependent on $F$ plays a crucial role.

\begin{defin} Let $D:\cat B \to\cat E$ be a finite product preserving functor and define the category $\cat E\downmono D$ as follows. The objects are monomorphisms $m:X\to DY$. Let $m':X'\to DY'$ be another object, then a morphism $m \to m'$ is a pair $(f,g)$ where $f:X\to X'$ of $\cat E$, $g:Y\to Y'$ of $\cat B$ and $Dg\circ m = m'\circ f$. \end{defin}

In this case $\cat E\downmono D$ is just a subcategory if the comma category $\cat E\downarrow D$ (in contrast to $D\downepi \cat E$ from the proof of lemma \ref{converse}). It comes with two projections $\Pi_0:\cat E\downmono D \to\cat E$ and $\Pi_1:\cat E\downmono D \to\cat B$. For a tripos $T$ resolvency comes into play for one of these projections.

\begin{lemma} If $T$ is a tripos, then the projection $\Pi_0:\cat B[T]\downmono \nabla \to\cat B[T]$ is resolvent.\label{tripco} \end{lemma}

\begin{proof} Let $(X,e)$ be an object of $\cat B[T]$. Define $\Sigma(e) \in T(\pi X)$ as follows:
\[ \Sigma(e)(\xi) = \im{x\in X} \forall_{y\in Y} \epsilon_X(\xi,y) \leftrightarrow e(x,y)\]
Let $\Sigma(X,e) = (\pi X, \im{\delta_X}(\Sigma(e)))$. This way, $\im{\delta_X}(\Sigma(e))$ defines a monomorphism $\Sigma(X,e) \to \nabla X$ and $\epsilon_X$ defines a morphism $\Sigma(X,e) \to (X,e)$.

Let $m:(Y,e') \to \nabla Z$ be some monomorphism and let $f:(Y,e') \to (X,e)$. The relation $f\circ m\inv$ equals $T(g\times \id_X)(\epsilon_X)$ for some $g:Z\to \pi X$ by the definition of triposes.
Now $\nabla g\circ m$ factors (uniquely) through the inclusion $\Sigma(X,e) \to \nabla\pi X$.
Furthermore, $\epsilon_X\circ \nabla g\circ m = f$. Therefore the monomorphism $\Sigma(X,e) \to \nabla\pi X$ is a resolution of $(X,e)$. By further generalization $\Pi_0:(\cat B[T]\downmono \nabla)\to\cat B[T]$ is resolvent. \end{proof} %narekenen!

\begin{remark} This lemma is another way of saying that every object $(X,e_X)$ in $\cat B[T]$ is isomorphic to a \emph{weakly complete object} $(X',e_{X'})$. A weakly complete object has the property that every functional relation $f:(Y,e_Y) \to (X', e_{X'})$ satisfies $e_Y\leq T(\id\times g)(f)$ for some $g:Y\to X'$. Tripos theory uses this property to explain why certain morphisms of triposes induce left exact functors between the represented toposes \citep{a2CAotTtTC,Pittsthesis,MR578267,MR2479466}.
\end{remark}

In a \emph{well powered} finitely complete category $\cat E$ every object $X$ has a \emph{set} $S_X$ of monomorphism such that for every monomorphism $m:Y\to X$ there is a monomorphism $m':Y'\to X$ in $S_X$ and an isomorphism $f:Y\to Y'$ such that $m=m'\circ f$. If $\cat E$ is well powered, then there is a functor $\sub:\cat E\dual \to \Set$ which assigns a poset of subobjects to each object $X$ of $\cat E$; for each morphism $f:X\to Y$ pullbacks of monomorphisms along $f$ determine the map $\sub(f):\sub(Y)\to \sub(X)$. 

Any product preserving functor $D:\cat B \to\cat E$ induces a presheaf $\sub(D-)$ on $\cat B$ which satisfies part of the properties of a tripos trivially. The poset $\sub(DX)$ is a Heyting algebra, for each morphism $f:X\to Y$ of $\cat B$, $Df\inv:\sub(DY) \to \sub(DX)$ is a morphism of Heyting algebras which has both adjoints. These adjoints satisfy Beck-Chevalley over the pullback squares that $D$ preserves, which mean that the adjoint are stable under products. The resolvency condition proved in lemma \label{trico} insures that $\sub(D-)$ also has generic predicates.

\begin{lemma} Let $\cat E$ be a well powered topos, let $D:\cat B \to\cat E$ be product preserving and let $\Pi_0:\cat E\downmono D \to \cat E$ be resolvent, then $\sub(D-)$ is a tripos. \label{cotrip} \end{lemma}

\begin{proof} Let $t:1\to\Omega$ be the subobject classifier of $\cat E$. For each $X$ of $\cat B$, $\Omega^{DX}$ has a resolution, which is a monomorphism $m:X_0\to D\pi X$ together with a morphism $e:X_0 \to DX$. Because $\Omega^{DX}$ is injective, $e$ factors through $m$ in a new map $c: D\pi X \to \Omega^X$ which necessarily determines an new resolution.

There is a morphism $D(\pi X\times X) \to \Omega^{DX}\times DX \to \Omega$ defined by $(x,y)\mapsto c(x)(y)$. Let $\epsilon_X$ be the set of all pullbacks of $t:1\to \Omega$ along this morphism.

Each monomorphism $m:Z \to D(Y\times X)$ is the pullback of $t:1\to\Omega$ along some $f:D(X\times Y) \to \Omega$. The transpose $f^t: DY \to \Omega^{DX}$ of $f$ factors through $c:D\pi X \to \Omega^{DX}$, giving $f_0: Y \to \pi X$ in $\cat B$ such that $c\circ Df_0 = f$. The subobject generated by $m$ is the inverse image of $\epsilon_X$ along $Df_0$. Hence $\epsilon_X$ is a generic predicate and $\sub(D-)$ is a tripos by generalization.
\end{proof}

The following notions of isomorphism turn triposes over $\cat B$ and finite product preserving functors $\cat B \to \cat E$ into two groupoids. The mapping $D\mapsto \sub(D-)$ is an equivalence of these groupoids.%hier

\begin{defin} Let $T,U:\cat B\dual \to \Set$ be triposes. An \emph{isomorphism of triposes} $T\to U$ is a natural isomorphism $\mu:T\to U$ such that $\mu_X:TX \to UX$ preserves finitary joins. Triposes and isomorphisms of triposes together form the \emph{groupoid of triposes} over $\cat B$.

Let $\cat E$ and $\cat E'$ be toposes and let $D:\cat B\to\cat E$ and $D':\cat B\to\cat E'$ be finite product preserving functors. In this case an isomorphism $D\to D'$ is an equivalence of categories $F:\cat E\to \cat E'$ together with a natural isomorphism $\phi:FD \to D'$. \end{defin} 

\begin{theorem} The groupoids of triposes over $\cat B$ is equivalent to the groupoid of finite product preserving functors $D:\cat B \to \cat E$ to toposes, for which $\Pi_0:\cat E/D \to \cat E$ is resolvent. \end{theorem}

\begin{proof} There is a canonical a natural isomorphism $T \to \sub(\nabla -)$ for each tripos $T$ \citep{a2CAotTtTC}.
The rest of this proof determines a isomorphism from $\nabla:\cat B \to \cat B[\sub(D-)]$ to $D$.

Let $F:\cat B[\sub(D-)] \to\cat E$ map $(X,e \in \sub(D(X\times X)))$ to the quotient object of the partial equivalence relation $e$ on $DX$; each functional relation $f:(X,e) \to (Y,e')$ induces a unique morphism $Ff:F(X,e) \to F(Y,e')$. Then $F$ is a fully faithful functor because each morphism $F(X,e) \to F(Y,e')$ determines a unique functional relation $(X,e) \to (Y,e')$. Resolvency implies that every object $X$ in $\cat E$ is a subquotient of some $DY$. Therefore the functor is also essentially surjective on objects and therefore an equivalence of categories. Finally, for each object $X$ of $\cat B$, $F\nabla X$ is the quotient of $DX$ by equality and therefore isomorphic to $DX$. For $f:X\to Y$ in $\cat B$, the functional relation $\nabla f$ is the graph of $Df$. This induces a morphism $F\nabla X \to F\nabla Y$ which commutes with $Df$ and the isomorphisms. Hence $\nabla$ and $D$ are isomorphic.
\end{proof}

\subsection{Conclusion} 
For each tripos $T:\cat B\dual \to \Set$ let $\Asm(T)$ be the full subcategory of $\cat B[T]$ on objects that have a monomorphism to $\nabla X$ for some object $X$ of $\cat B$. This too is the full image of $\Pi_0:\cat B[T]\downmono \nabla \to \cat B[T]$. It is a regular category because it is closed under finite products and arbitrary subobjects, which include all equalizers and images. The embedding $\Asm(T) \to \cat B[T]$ is resolvent, because $\Pi_0$ is and because the factorization of $\Pi_0$ through $\Asm(T)$ is surjective on objects.

It is easy to see that $\cat B[T] \cong \Asm(T)\exreg$, using the fact that $J:\Asm(T) \to \Asm(T)\exreg$ is the regular completion to get a functor $\Asm(T)\exreg \to \cat B[T]$ and using the fact that the inclusion $\Asm(T)\to \cat B[T]$ is resolvent to find its inverse $\cat B[T] \to \Asm(T)\exreg$.
Hence $\Asm(T)$ is weakly locally Cartesian closed and has a generic monomorphism.

In the case of the effective tripos $E:\Set\dual \to \Set$, $\Eff$ \emph{is} $\Set[E]$ and $\Asm(E)$ is (equivalent to) the ordinary category of assemblies $\Asm$. This solves the mystery at the start of this paper: if $\cat E$ is exact then finitely continuous functors $\Asm \to\cat E$ have left Kan extensions along the inclusion $\Asm\to \Eff$.

\hide{ Analyse: morfismes van tripossen. Probleem: $\Tot(T)$ is niet finitely complete. Misschien iets voor een ander artikel dus. }

\subsection*{Acknowledgments} 
I am grateful to the Warsaw Center of Mathematics and Computer Science for the opportunity to write this paper, for Mat\'ias Menni's many useful comments and for discussions with Marek Zawadowski.

\bibliographystyle{plainnat}

\bibliography{realizability}{}

\end{document}